\documentclass[a4paper,oneside,final]{article}%
\usepackage{amsmath}
\usepackage{graphicx}
\usepackage{amsfonts}
\usepackage{amssymb}%
\providecommand{\U}[1]{\protect\rule{.1in}{.1in}}
\providecommand{\U}[1]{\protect\rule{.1in}{.1in}}
\newtheorem{theorem}{Theorem}[section]

\newtheorem{definition}[theorem]{Definition}

\newtheorem{lemma}[theorem]{Lemma}

\newtheorem{proposition}[theorem]{Proposition}
\newtheorem{remark}[theorem]{Remark}

\newenvironment{proof}[1][Proof]{\textbf{#1.} }{\ \rule{0.5em}{0.5em}}
\begin{document}

\title{Stochastic Differential Equations Driven by Fractional Brownian Motion and
Standard Brownian Motion}
\author{Jo\~{a}o Guerra\thanks{Supported by FCT- (FEDER/POCI 2010) and by the MCyT
Grant No. BFM2003-04294}\\CEMAPRE, Rua do Quelhas 6, \\1200-781 Lisboa, Portugal. \\E-mail: jguerra@iseg.utl.pt
\and David Nualart \thanks{Supported by the NSF grant DMS0604207}\\Department of Mathematics, University of Kansas, \\405 Snow Hall, Lawrence, Kansas, 66045-2142\\E-mail: nualart@math.ku.edu}
\maketitle

\begin{abstract}
We prove an existence and uniqueness theorem for solutions of
multidimensional, time dependent, stochastic differential equations driven
simultaneously by a multidimensional fractional Brownian motion with Hurst
parameter $H>1/2$ and a multidimensional standard Brownian motion. The proof
relies on some a priori estimates, which are obtained using the methods of
fractional integration, and the classical It\^{o} stochastic calculus. The
existence result is based on the Yamada-Watanabe theorem.



\end{abstract}

\section{Introduction}

The fractional Brownian motion (fBm) with Hurst parameter $H\in(0,1)$ is a
zero mean Gaussian process $B^{H}=\{B_{t}^{H},t\geq0\}$ with covariance
function%
\begin{equation}
R_{H}(s,t)=\frac{1}{2}\left(  t^{2H}+s^{2H}-\left\vert t-s\right\vert
^{2H}\right)  . \label{Covfbm}%
\end{equation}
This process was introduced by Kolmogorov in \cite{Kolmogorov} and later
studied by Mandelbrot and Van Ness in \cite{Mandelbrot}. Its self-similar and
long-range dependence (if $H>\frac{1}{2}$) properties make this process a
useful driving noise in models arising in physics, telecommunication networks,
finance and other fields. For a review of some applications of fBm we refer to
\cite{Decreusefond2}.

From (\ref{Covfbm}) we deduce that $\mathbb{E}(\left|  B_{t}^{H}-B_{s}%
^{H}\right|  ^{2})=|t-s|^{2H}$ and, as a consequence, the trajectories of
$B^{H}$ are almost surely locally $\alpha$-H\"{o}lder continuous for all
$\alpha\in(0,H)$. Since $B^{H}$ is not a semimartingale if $H\neq1/2$ (see
\cite{Rogers}),\ we cannot use the classical It\^{o} theory to construct a
stochastic calculus with respect to the fBm. Over the last years some new
techniques have been developed in order to define stochastic integrals with
respect to fBm. Essentially two different types of integrals can be defined:

\begin{itemize}
\item[i)] The Skorokhod integral (or divergence integral) with respect to fBm
is defined as the adjoint of the derivative operator in the framework of the
Malliavin calculus. This approach was introduced by Decreusefond and
\"{U}st\"{u}nel \cite{Decreusefond}, and later developed by Carmona and Coutin
\cite{Carmona}, Duncan, Hu and Pasik-Duncan \cite{Duncan}, Al\`{o}s, Mazet and
Nualart \cite{Alos}, Hu and \O ksendal \cite{Hu}, Cheridito and Nualart
\cite{CN}, among others. This stochastic integral satisfies the zero mean
property and it \ can be expressed as the limit of Riemann sums defined using
Wick products.

\item[ii)] The pathwise Riemann-Stieltjes integral $\int_{0}^{T}v_{s}%
dB_{s}^{H}$ exists if $\left\{  v_{t},t\in\left[  0,T\right]  \right\}  _{\ }$
is a stochastic process with H\"{o}lder continuous paths of order $\alpha
>1-H$, as a consequence of the results of Young \cite{Young}. Z\"{a}hle
expressed this integral in terms of fractional derivative operators, using the
fractional integration by parts formula (see \cite{Zahle}). We also refer to
\cite{RV} for a pathwise approach to the stochastic calculus based on the
regularization of the noise.
\end{itemize}

The aim of this paper is to study the $d$-dimensional stochastic differential
equation
\begin{equation}
X_{t}=X_{0}+\int_{0}^{t}b(s,X_{s})\mathrm{d}s+\int_{0}^{t}\sigma_{W}%
(s,X_{s})\mathrm{d}W_{s}+\int_{0}^{t}\sigma_{H}(s,X_{s})\mathrm{d}B_{s}^{H},
\label{sde1}%
\end{equation}
where $W$ is an $r$-dimensional\ standard Brownian motion and $B^{H}$ is an
$m$-dimensional fractional Brownian motion with $\ H\in\left(  \frac{1}%
{2},1\right)  $. We assume that the processes $W$ and $B^{H}$ are independent.
In the above stochastic differential equation, the integral $\int_{0}%
^{t}\sigma_{W}(s,X_{s})\mathrm{d}W_{s}$ should be interpreted as an It\^{o}
stochastic integral and the integral $\int_{0}^{t}\sigma_{H}(s,X_{s}%
)\mathrm{d}B_{s}$ as a pathwise Riemann-Stieltjes integral in the sense of
Z\"{a}hle \cite{Zahle}. Our main result is a general theorem about the
existence and uniqueness of solutions for the stochastic differential equation
(\ref{sde1}) under suitable conditions on the coefficients.

Equations driven only by a fBm with Hurst parameter $H\in(\frac{1}{2},1)$ can
be solved by a pathwise approach using the $p$-variation norm (see
\cite{Lyons2}), the fractional calculus introduced by Z\"{a}hle (see
\cite{Zahle} and \cite{Nualart}), or H\"{o}lder norms \cite{Ruzmaikina}. Also
using the tools of rough path analysis introduced by Lyons in \cite{Lyons},
Coutin and Qian proved in \cite{Coutin} the existence of strong solutions for
stochastic differential equations driven by fBm with $H>\frac{1}{4}$ and
studied a Wong-Zakai approximation limit for these stochastic differential equations.

Kubilius has studied stochastic differential equations driven by both fBm and
standard Brownian motion (see \cite{KubiliusW}), in the one dimensional case,
with $\sigma_{W}$, $\sigma_{H}$ independent of the time and with no drift term
($b\equiv0$). In this setting, the author proves an existence and uniqueness
result provided that $\sigma_{W}$ is a Lipschitz function and $\sigma_{H}\in
C^{1+\delta}$, with $\delta>q\left(  1-H\right)  $\thinspace, $q>2$. With
these assumptions, the solution belongs to the space of continuous functions
with $q$-bounded variation. Kubilius defines the stochastic integral with
respect to fBm as an extended Riemann-Stieltjes pathwise integral and he uses
$p$-variation estimates.

Our approach is completely different from Kubilius \cite{KubiliusW} in the
sense that we combine the pathwise approach with the It\^{o} stochastic
calculus in order to handle both types of integrals. Then, the uniqueness of a
solution follows from estimates for both It\^{o} and Riemann-Stieltjes
integrals. However, the existence of a strong solution cannot be obtained by
the classical fixed point argument because the estimates of the H\"{o}lder
norm of an integral with respect to $B^{H}$ produce some higher order terms.
For this reason, we first prove the existence of weak solutions and later
deduce the existence of strong solutions using the Yamada-Watanabe theorem.

The paper is organized as follows. In Section 2 we state the problem and list
our assumptions on the coefficients of Eq. (\ref{sde1}). Section 3 provides
some estimates for fractional and It\^{o} integrals. In section 4, the
pathwise uniqueness property of solutions of Eq. (\ref{sde1}) is proved. In
section 5, we introduce the Euler sequence that approximates the solution of
the stochastic differential equation and prove that it is a tight sequence.
Then, the Skorokhod representation theorem is applied in order to prove the
existence of a weak solution for the stochastic differential equation.
Finally, we prove the existence of a unique strong solution by using the
Yamada-Watanabe theorem.

\setcounter{equation}{0}

\section{Preliminaries}

Fix a time interval $[0,T]$ and a complete probability space $\left(
\Omega,\mathcal{F},P\right)  $. Suppose that $B^{H}=\{B_{t}^{H},t\in
\lbrack0,T]\}$ is an $m$-dimensional fractional Brownian motion with Hurst
parameter $H\in\left(  \frac{1}{2},1\right)  $, and $W=\{W_{t},t\in
\lbrack0,T]\}$ an $r$-dimensional standard Brownian motion, independent of
$B^{H}$. Let $X_{0}$ be a $d$-dimensional random variable independent of
$(B^{H},W)$. For each $t\in\lbrack0,T]$ we denote by $\mathcal{F}_{t}$ the
$\sigma$-field generated by the random variables $\{X_{0},B_{s}^{H},$
$W_{s},s\in\lbrack0,t]\}$ and the $P$-null sets.

In addition to the natural filtration $\{\mathcal{F}_{t},t\in\lbrack0,T]\}$ we
will consider bigger filtrations $\{\mathcal{G}_{t},t\in\lbrack0,T]\}$ such that:

\begin{enumerate}
\item $\left\{  \mathcal{G}_{t}\right\}  $ is right-continuous and
$\mathcal{G}_{0}$ contains the $P$-null sets.

\item $X_{0}$ and $B^{H}$ are $\mathcal{G}_{0}$-measurable, and $W$ is a
$\mathcal{G}_{t}$-Brownian motion.
\end{enumerate}

Notice that $\widehat{\mathcal{F}}_{t}\subset\mathcal{G}_{t}$, where
$\widehat{\mathcal{F}}_{t}$ is the $\sigma$-field generated by the random
variables $\{X_{0},B^{H},$ $W_{s},s\in\lbrack0,t]\}$ and the $P$-null sets.

Consider the stochastic differential equation (\ref{sde1}), where $X_{0}$ is a
$d$-dimensional random variable independent of $(B^{H},W)$ and the
coefficients are measurable functions $\ b^{i},\sigma_{W}^{i,k},\sigma
_{H}^{i,j}:\left[  0,T\right]  \times\mathbb{R}^{d}\rightarrow\mathbb{R}$,
$1\leq i\leq d$, $1\leq k\leq r$, $1\leq j\leq m$. We will make use of the
following assumptions on the coefficients.

\begin{itemize}
\item[(H$b$)] The function $b(t,x)$ is continuous. Moreover, it is Lipschitz
continuous in the variable $x$ and has linear growth in the same variable,
uniformly in $t$, that is, there exist constants $L_{1}$ and $L_{2}$ such that%
\begin{align*}
\left\vert b(t,x)-b(t,y)\right\vert  &  \leq L_{1}\left\vert x-y\right\vert
,\\
\left\vert b(t,x)\right\vert  &  \leq L_{2}\left(  1+\left\vert x\right\vert
\right)  ,
\end{align*}
for all $x,y\in\mathbb{R}^{d}\ $\ and $t\in\left[  0,T\right]  $.

\item[(H$\sigma_{W}$)] The function $\sigma_{W}(t,x)$ is continuous. Moreover,
it is Lipschitz continuous in $x$ and has linear growth in the same variable,
uniformly in $t$, that is, there exist constants $L_{3}$ and $L_{4}$ such
that
\begin{align*}
\left\vert \sigma_{W}(t,x)-\sigma_{W}(t,y)\right\vert  &  \leq L_{3}\left\vert
x-y\right\vert ,\\
\left\vert \sigma_{W}(t,x)\right\vert  &  \leq L_{4}\left(  1+\left\vert
x\right\vert \right)  ,
\end{align*}
for all $x,y\in\mathbb{R}^{d}\ $\ and $t\in\left[  0,T\right]  $.

\item[(H$\sigma_{H})$] The function $\sigma_{H}(t,x)$ is continuous and
continuously differentiable in the variable $x$. Moreover,\ there exist
constants $L_{5}$, $L_{6}$ and $L_{7}$ such that%
\begin{align*}
\left\vert \partial_{x_{i}}\sigma_{H}(t,x)\right\vert  &  \leq L_{5},\\
\left\vert \partial_{x_{i}}\sigma_{H}(t,x)-\partial_{x_{i}}\sigma
_{H}(t,y)\right\vert  &  \leq L_{6}\left\vert x-y\right\vert ^{\delta},\\
\left\vert \sigma_{H}(t,x)-\sigma_{H}(s,x)\right\vert +\left\vert
\partial_{x_{i}}\sigma_{H}(t,x)-\partial_{x_{i}}\sigma_{H}(s,x)\right\vert  &
\leq L_{7}\left\vert t-s\right\vert ^{\beta},
\end{align*}
for all $x,y\in\mathbb{R}^{d}\ $\ and $t\in\left[  0,T\right]  $, and for some
constants $0<\delta,\beta\leq1$.
\end{itemize}

Note that assumption $\ $(H$\sigma_{H}$) implies the linear growth property,
i. e., there exists a constant $L$ such that
\begin{equation}
\left|  \sigma_{H}(t,x)\right|  \leq L\left(  1+\left|  x\right|  \right)
\label{HsigmaH growth}%
\end{equation}
for all $x,y\in\mathbb{R}^{d}\ $\ and $t\in\left[  0,T\right]  $.

Let us now introduce some function spaces that will be used in the analysis of
solutions\ of the stochastic differential equation (\ref{sde1}). Let
$0<\alpha<\frac{1}{2}$. For any measurable function $f:\left[  0,T\right]
\rightarrow\mathbb{R}^{d}$ we introduce the following notation
\begin{equation}
\left\|  f(t)\right\|  _{\alpha}:=\left|  f(t)\right|  +\int_{0}^{t}%
\frac{\left|  f(t)-f(s)\right|  }{\left(  t-s\right)  ^{\alpha+1}}\mathrm{d}s.
\label{alpha}%
\end{equation}
Denote by $W_{0}^{\alpha,\infty}$ the space of measurable functions $f:\left[
0,T\right]  \rightarrow\mathbb{R}^{d}$ such that%

\begin{equation}
\left\|  f\right\|  _{\alpha,\infty}:=\underset{t\in\left[  0,T\right]  }%
{\sup}\left\|  f(t)\right\|  _{\alpha}<\infty. \label{fnorm}%
\end{equation}
For $\mu\in(0,1]$ let $C^{\mu}$ be space of $\mu$-H\"{o}lder continuous
functions $f:\left[  0,T\right]  \rightarrow\mathbb{R}^{d}$, equipped with the norm%

\begin{equation}
\left\Vert f\right\Vert _{\mu}:=\left\Vert f\right\Vert _{\infty}%
+\underset{0\leq s<t\leq T}{\sup}\frac{\left\vert f(t)-f(s)\right\vert
}{\left(  t-s\right)  ^{\mu}}<\infty, \label{Hnorm}%
\end{equation}
where $\left\Vert f\right\Vert _{\infty}:=\underset{t\in\left[  0,T\right]
}{\sup}\left\vert f(t)\right\vert .$ Given any $\varepsilon$ such that
$0<\varepsilon<\alpha$, we have the following inclusions:%

\begin{equation}
C^{\alpha+\varepsilon}\subset W_{0}^{\alpha,\infty}\subset C^{\alpha
-\varepsilon}. \label{Eq inclusion Holder W}%
\end{equation}
In particular, both the fractional Brownian motion $B^{H}$, with $H>\frac
{1}{2},$ and the standard Brownian motion $W,$ have their trajectories in
$W_{0}^{\alpha,\infty}$. We denote by $\mathbb{E}^{W}$ the conditional
expectation given $\widehat{\mathcal{F}}_{0}$, that is, given $X_{0}$ and
$B^{H}$.

We now define the space of processes where we will search for solutions of
(\ref{sde1}).

\begin{definition}
\label{Def SIGMA SOL}Let $\mathbb{W}_{\mathcal{G}}$ be the space of
$d$-dimensional $\mathcal{G}_{t}$-adapted stochastic processes $X=\left\{
X_{t},t\in\left[  0,T\right]  \right\}  $ such that \ almost surely the
trajectories of $X$ belong to $W_{0}^{\alpha,\infty}$ and $\int_{0}%
^{T}\mathbb{E}^{W}\left[  \left\|  X_{s}\right\|  _{\alpha}^{2}\right]
\mathrm{d}s<\infty.$
\end{definition}

A strong solution of the stochastic differential equation (\ref{sde1}) is a
stochastic process $X$ in the space $\mathbb{W}_{\mathcal{F}}$, which
satisfies (\ref{sde1}) a.s. The main result proved in this paper is the
following theorem on the uniqueness and existence of \ strong solutions for
(\ref{sde1}).

\begin{theorem}
\label{main}Assume that the coefficients $b,$ $\sigma_{W}$ and $\sigma_{H}$
satisfy the assumptions (H$b$), (H$\sigma_{W}$) and (H$\sigma_{H}$). If
$\alpha$ satisfies $1-H<\alpha<\min\left\{  \frac{1}{2},\beta,\frac{\delta}%
{2}\right\}  $, then there exists a unique strong solution\ $X$ of Equation
(\ref{sde1}).
\end{theorem}

\noindent\textbf{Remark} \ Notice that in all our results we can replace the
fractional Brownian motion $B^{H}$ by an arbitrary stochastic process with
H\"{o}lder continuous trajectories of order $\gamma>\frac{1}{2}$.

\setcounter{equation}{0}

\section{Integral estimates}

In this section we will first define the integral with respect to fBm as a
generalized Stieltjes integral, following the work of Z\"{a}hle \cite{Zahle}.
We also present some basic estimates of this integral.

Let $f\in L^{1}(a,b)$ and $\alpha>0$. The left-sided and right-sided
fractional Riemann-Liouville integrals of $f$ of order $\alpha$ are defined
for almost all $x\in(a,b)$ by
\[
I_{a^{+}}^{\alpha}f(x):=\frac{1}{\Gamma(\alpha)}\int_{a}^{x}(x-y)^{\alpha
-1}f(y)\mathrm{d}y
\]
and
\[
I_{b^{-}}^{\alpha}f(x):=\frac{1}{\Gamma(\alpha)}\int_{x}^{b}(y-x)^{\alpha
-1}f(y)\mathrm{d}y
\]
respectively, where $\Gamma(\alpha):=\int_{0}^{\infty}r^{\alpha-1}%
e^{-r}\mathrm{d}r$ is the Euler gamma function. Let $I_{a^{+}}^{\alpha}%
(L^{p})$ (resp. $I_{b^{-}}^{\alpha}(L^{p})$) be the image of $L^{p}(a,b)$ by
the operator $I_{a^{+}}^{\alpha}$(resp. $I_{b^{-}}^{\alpha}$). If $f\in
I_{a^{+}}^{\alpha}(L^{p})$ (resp. $f\in I_{b^{-}}^{\alpha}(L^{p})$) and
$0<\alpha<1$ then the Weyl derivatives of $f$ are given by
\begin{equation}
D_{a^{+}}^{\alpha}f(x):=\frac{1}{\Gamma(1-\alpha)}\left(  \frac{f(x)}{\left(
x-a\right)  ^{\alpha}}+\alpha\int_{a}^{x}\frac{f(x)-f(y)}{(x-y)^{\alpha+1}%
}\mathrm{d}y\right)  1_{(a,b)}(x) \label{Fractional derivative a}%
\end{equation}
and%
\begin{equation}
D_{b^{-}}^{\alpha}f(x):=\frac{1}{\Gamma(1-\alpha)}\left(  \frac{f(x)}{\left(
b-x\right)  ^{\alpha}}+\alpha\int_{x}^{b}\frac{f(x)-f(y)}{(y-x)^{\alpha+1}%
}\mathrm{d}y\right)  1_{(a,b)}(x), \label{Fractional derivative b}%
\end{equation}
respectively, and are defined for almost all $x\in(a,b)$ (the convergence of
the integrals at the singularity $y=x$ holds pointwise for almost all
$x\in(a,b)$ if $p=1$ and moreover in $L^{p}$-sense if $1<p<\infty$).

We have that:

\begin{itemize}
\item If $\alpha<\frac{1}{p}$ and $q=\frac{p}{1-\alpha p}$ then $I_{a^{+}%
}^{\alpha}(L^{p})=I_{b^{-}}^{\alpha}(L^{p})\subset L^{q}(a,b).$

\item If $\alpha>\frac{1}{p}$ then $I_{a^{+}}^{\alpha}(L^{p})\cup I_{b^{-}%
}^{\alpha}(L^{p})\subset C^{\alpha-\frac{1}{p}}(a,b).$
\end{itemize}

The fractional integrals and derivatives are related by the inversion formulas%
\[
I_{a^{+}}^{\alpha}(D_{a^{+}}^{\alpha}f)=f,\ \ \forall f\in I_{a^{+}}^{\alpha
}(L^{p}),
\]%
\[
D_{a^{+}}^{\alpha}(I_{a^{+}}^{\alpha}f)=f,\ \ \forall f\in L^{1}(a,b),
\]
and similar formulas also hold for $I_{b^{-}}^{\alpha}$ and $D_{b^{-}}%
^{\alpha}$. We refer to \ \cite{Samko} for a detailed account on the
properties of fractional operators.

Let $f(a+):=\underset{\varepsilon\searrow0}{\lim}f(a+\varepsilon)$ and
$g(b-):=\underset{\varepsilon\searrow0}{\lim}g(b-\varepsilon)$ (we are
assuming that these limits exist and are finite) and define%

\[
f_{a^{+}}(x):=\left(  f(x)-f(a+)\right)  1_{(a,b)}(x),
\]%
\[
g_{b^{-}}(x):=\left(  g(x)-g(b-)\right)  1_{(a,b)}(x).
\]
We recall from \cite{Zahle} the definition of generalized Stieltjes
\ fractional integral with respect to irregular functions.

\begin{definition}
\textbf{(Generalized Stieltjes Integral) }Suppose that $f$ and $g$ are
functions such that $f(a+),g(a+)$ and $g(b-)$ exist, $f_{a^{+}}\in I_{a^{+}%
}^{\alpha}(L^{p})$ and $g_{b^{-}}\in$ $I_{b^{-}}^{1-\alpha}(L^{p}) $ for some
$p,q\geq1,1/p+1/q\leq1,0<\alpha<1$. Then the integral of $f$ with respect to
$g$ is defined by
\[
\int_{a}^{b}f\mathrm{d}g:=(-1)^{\alpha}\int_{a}^{b}D_{a^{+}}^{\alpha
}f(x)D_{b^{-}}^{1-\alpha}g_{b^{-}}(x)\mathrm{d}x+f(a+)\left(
g(b-)-g(a+)\right)  .
\]

\end{definition}

\begin{remark}
The above definition is simpler in the following cases.

\begin{itemize}
\item If $\alpha p<1$, under the assumptions of the preceding definition, we
have that $f\in I_{a^{+}}^{\alpha}(L^{p})$ and we can write
\begin{equation}
\int_{a}^{b}f\mathrm{d}g=(-1)^{\alpha}\int_{a}^{b}D_{a^{+}}^{\alpha}f_{a^{+}%
}(x)D_{b^{-}}^{1-\alpha}g_{b^{-}}(x)\mathrm{d}x. \label{DefRS}%
\end{equation}

\item If $f\in C^{\lambda}(a,b)$ and $g\in C^{\mu}(a,b)$ with $\lambda+\mu>1$
then (see \cite{Zahle}) we can choose $\alpha$ such that $1-\mu<\alpha
<\lambda$, the generalized Stieltjes integral exists, it is given by
(\ref{DefRS}) and coincides with the Riemann-Stieltjes integral.
\end{itemize}
\end{remark}

The linear spaces $I_{a^{+}}^{\alpha}(L^{p})$ are Banach spaces with respect
to the norms
\[
\left\|  f\right\|  _{I_{a^{+}}^{\alpha}(L^{p})}:=\left\|  f\right\|  _{L^{p}%
}+\left\|  D_{a^{+}}^{\alpha}f\right\|  _{L^{p}}\sim\left\|  D_{a^{+}}%
^{\alpha}f\right\|  _{L^{p}},
\]
and the same is true for the spaces $I_{b^{-}}^{\alpha}(L^{p})$. If $\alpha
p<1$ then the norms on $I_{a^{+}}^{\alpha}(L^{p})$ and $I_{b^{-}}^{\alpha
}(L^{p})$ are equivalent and if $a\leq c<d\leq b$, then
\[
\int_{c}^{d}f\mathrm{d}g:=\int_{a}^{b}\mathbf{1}_{(c,d)}f\mathrm{d}g.
\]

Now, fix the parameter $\alpha$ such that $0<\alpha<\frac{1}{2}$, denote by
$W_{T}^{1-\alpha,\infty}$ the space of measurable functions $g:\left[
0,T\right]  \rightarrow\mathbb{R}^{m}$ such that
\[
\left\|  g\right\|  _{1-\alpha,\infty,T}:=\underset{0<s<t<T}{\sup}\left(
\frac{\left|  g(t)-g(s)\right|  }{(t-s)^{1-\alpha}}+\int_{s}^{t}\frac{\left|
g(y)-g(s)\right|  }{(y-s)^{2-\alpha}}\mathrm{d}y\right)  <\infty.
\]
and denote by $W_{0}^{\alpha,1}$the space of measurable functions $f:\left[
0,T\right]  \rightarrow\mathbb{R}^{d}$ such that
\[
\left\|  f\right\|  _{\alpha,1}:=\int_{0}^{T}\frac{\left|  f(s)\right|
}{s^{\alpha}}\mathrm{d}s+\int_{0}^{T}\int_{0}^{s}\frac{\left|
f(s)-f(y)\right|  }{(s-y)^{\alpha+1}}\mathrm{d}y\mathrm{d}s<\infty.
\]
It is easy to prove that
\[
C^{1-\alpha+\varepsilon}\subset W_{T}^{1-\alpha,\infty}\subset C^{1-\alpha},
\]
for all $\varepsilon>0$. For $g\in W_{T}^{1-\alpha,\infty}$, we have that
\begin{align*}
\Lambda_{\alpha}(g)  &  :=\frac{1}{\Gamma(1-\alpha)}\underset{0<s<t<T}{\sup
}\left|  \left(  D_{t-}^{1-\alpha}g_{t-}\right)  (s)\right| \\
&  \leq\frac{1}{\Gamma(1-\alpha)\Gamma(\alpha)}\left\|  g\right\|
_{1-\alpha,\infty,T}<\infty.
\end{align*}
Moreover, if $f\in W_{0}^{\alpha,1}$ and $g\in$ $W_{T}^{1-\alpha,\infty}$ then
$\int_{0}^{t}f\mathrm{d}g$ exists for all $t\in\left[  0,T\right]  $ and
\begin{equation}
\left|  \int_{0}^{t}f\mathrm{d}g\right|  \leq\Lambda_{\alpha}(g)\left\|
f\right\|  _{\alpha,1}. \label{Gfa1}%
\end{equation}

Now we will deduce useful estimates for the integrals involved in Equation
(\ref{sde1}). Fix $\alpha\in(1-H,\frac{1}{2})$. We will denote by $C$ a
generic constant which depends on the constants $L_{i}$, $1\leq i\leq7$,
$\beta$ and $\delta$ in the assumptions, on $T$, $\alpha$ and the dimensions
$r,d,m$. For any function $f\in W_{0}^{\alpha,\infty}$ define
\[
F_{t}^{b}(f):=\int_{0}^{t}b(s,f(s))\mathrm{d}s.
\]

\begin{proposition}
\label{propFbf} \ If $f\in W_{0}^{\alpha,\infty}$ then $F^{b}\left(  f\right)
\in W_{0}^{\alpha,\infty}$ and for all $t\in\lbrack0,T]$
\begin{equation}
\left\|  F_{t}^{b}(f)\right\|  _{\alpha}\leq C\left(  \int_{0}^{t}%
\frac{\left|  f(s)\right|  }{(t-s)^{\alpha}}\mathrm{d}s+1\right)  .
\label{Fbf}%
\end{equation}

\end{proposition}

\begin{proof}
By Proposition 4.3 in \cite{Nualart} and the growth assumption in (H$b$) we
have that
\begin{align*}
\left\Vert F_{t}^{b}(f)\right\Vert _{\alpha}  &  \leq C\int_{0}^{t}%
\frac{\left\vert b(s,f(s))\right\vert }{(t-s)^{\alpha}}\mathrm{d}s\\
&  \leq C\int_{0}^{t}\frac{\left\vert f(s)\right\vert +1}{(t-s)^{\alpha}%
}\mathrm{d}s\\
&  \leq C\left(  \int_{0}^{t}\frac{\left\vert f(s)\right\vert }{(t-s)^{\alpha
}}\mathrm{d}s+1\right)  .
\end{align*}

\end{proof}

\begin{proposition}
If $f$, $h\in W_{0}^{\alpha,\infty}$ then for all $t\in\lbrack0,T]$
\begin{equation}
\left\Vert F_{t}^{b}(f)-F_{t}^{b}(h)\right\Vert _{\alpha}\leq C\int_{0}%
^{t}\frac{\left\Vert f(s)-h(s)\right\Vert _{\alpha}}{\left(  t-s\right)
^{\alpha}}\mathrm{d}s. \label{Fbfh}%
\end{equation}

\end{proposition}

\begin{proof}
By Proposition 4.3 in \cite{Nualart} and the Lipschitz assumption in (H$b$),
we have that
\begin{align*}
\left\Vert F_{t}^{b}(f)-F_{t}^{b}(h)\right\Vert _{\alpha}  &  \leq C\int
_{0}^{t}\frac{\left\vert b(s,f(s)-b(s,h(s))\right\vert }{(t-s)^{\alpha}%
}\mathrm{d}s\\
&  \leq C\int_{0}^{t}\frac{\left\vert f(s)-h(s)\right\vert }{\left(
t-s\right)  ^{\alpha}}\mathrm{d}s.
\end{align*}

\end{proof}

Given a function $f\in W_{0}^{\alpha,\infty}$, let us define
\[
G_{t}^{\sigma H}(f):=\int_{0}^{t}\sigma_{H}(s,f(s))\mathrm{d}B_{s}^{H}.
\]

\begin{proposition}
Suppose $1-H<\alpha<\min(\frac{1}{2},\beta)$. Then for all $t\in\lbrack0,T]$
\begin{equation}
\left\|  G_{t}^{\sigma H}(f)\right\|  _{\alpha}\leq C\Lambda_{\alpha}%
(B^{H})\int_{0}^{t}\left(  (t-s)^{-2\alpha}\ +s^{-\alpha}\right)  \left(
1+\left\|  f(s)\right\|  _{\alpha}\right)  \mathrm{d}s. \label{GsigmaHf2}%
\end{equation}

\end{proposition}

\begin{proof}
By Proposition 4.1 of \cite{Nualart} and the H\"{o}lder continuity in time,
given in assumption (H$\sigma_{H}$), we have%
\begin{align*}
\left\Vert G_{t}^{\sigma H}(f)\right\Vert _{\alpha}  &  \leq C\Lambda_{\alpha
}(B^{H})\int_{0}^{t}\left(  (t-s)^{-2\alpha}\ +s^{-\alpha}\right)  \left\Vert
\sigma_{H}(s,f(s))\right\Vert _{\alpha}\mathrm{d}s\\
&  \leq C\Lambda_{\alpha}(B^{H})\int_{0}^{t}\left(  (t-s)^{-2\alpha
}\ +s^{-\alpha}\right)  \left(  1+\left\Vert f(s)\right\Vert _{\alpha}\right)
\mathrm{d}s.
\end{align*}

\end{proof}

\begin{proposition}
\label{PropGhfh}If $1-H<\alpha<\min\left(  \frac{1}{2},\beta\right)  $ and
$f$, $h\in W_{0}^{\alpha,\infty}$ then for all $t\in\lbrack0,T]$
\begin{align}
&  \left\|  G_{t}^{\sigma H}(f)-G_{t}^{\sigma H}(h)\right\|  _{\alpha}\leq
C\Lambda_{\alpha}(B^{H})\label{GHfh}\\
&  \times\int_{0}^{t}\left(  (t-s)^{-2\alpha}\ +s^{-\alpha}\right)  \left(
1+\Delta f(s)+\Delta h(s)\right)  \left\|  f(s)-h(s)\right\|  _{\alpha
}\mathrm{d}s,\nonumber
\end{align}
\qquad where we denote
\begin{equation}
\Delta f(s):=\int_{0}^{s}\frac{\left|  f(s)-f(r)\right|  ^{\delta}}{\left(
s-r\right)  ^{\alpha+1}}\mathrm{d}r. \label{Delta}%
\end{equation}

\end{proposition}

\begin{proof}
From Proposition 4.1. of \cite{Nualart}, we have that
\begin{align*}
&  \left\|  G_{t}^{\sigma H}(f)-G_{t}^{\sigma H}(h)\right\|  _{\alpha}\leq
C\Lambda_{\alpha}(B^{H})\int_{0}^{t}\left(  (t-s)^{-2\alpha}\ +s^{-\alpha
}\right) \\
&  \times\left\|  \sigma_{H}(s,f(s))-\sigma_{H}(s,h(s))\right\|  _{\alpha
}\mathrm{d}s\\
&  \leq C\Lambda_{\alpha}(B^{H})\int_{0}^{t}\left(  (t-s)^{-2\alpha
}\ +s^{-\alpha}\right)  \left(  \left|  \sigma_{H}(s,f(s))-\sigma
_{H}(s,h(s))\right|  \right. \\
&  \left.  +\int_{0}^{s}\frac{\left|  \sigma_{H}(s,f(s))-\sigma_{H}%
(s,h(s))-\sigma_{H}(r,f(r))+\sigma_{H}(r,h(r))\right|  }{(s-r)^{\alpha+1}%
}\mathrm{d}r\right)  \mathrm{d}s.
\end{align*}
Now, using the assumptions in (H$\sigma_{H}$) and Lemma 7.1 in
\ \cite{Nualart}, we have that%
\begin{align*}
&  \left|  \sigma(t_{1},x_{1})-\sigma(t_{2},x_{2})-\sigma(t_{1},x_{3}%
)+\sigma(t_{2},x_{4})\right|  \leq\\
&  \leq L_{5}\left|  x_{1}-x_{2}-x_{3}+x_{4}\right|  +L_{7}\left|  x_{1}%
-x_{3}\right|  \left|  t_{2}-t_{1}\right|  ^{\beta}\\
&  +L_{6}\left|  x_{1}-x_{3}\right|  \left(  \left|  x_{1}-x_{2}\right|
^{\delta}+\left|  x_{3}-x_{4}\right|  ^{\delta}\right)  .
\end{align*}
As a consequence,
\begin{align*}
&  \left\|  G_{t}^{\sigma H}(f)-G_{t}^{\sigma H}(h)\right\|  _{\alpha}\leq
C\Lambda_{\alpha}(B^{H})\int_{0}^{t}\left(  (t-s)^{-2\alpha}\ +s^{-\alpha
}\right) \\
&  \mathbb{\times}\left(  1+\Delta f(s)+\Delta h(s)\right)  \left(  \left|
f(s)-h(s)\right|  +\int_{0}^{s}\frac{\left|  f(s)-h(s)-f(r)+h(r)\right|
}{(s-r)^{\alpha+1}}\mathrm{d}r\right)  \mathrm{d}s\\
&  \leq C\Lambda_{\alpha}(B^{H})\int_{0}^{t}\left(  (t-s)^{-2\alpha
}\ +s^{-\alpha}\right) \\
&  \times\left(  1+\left(  \Delta f(s)\right)  ^{2}+\left(  \Delta
h(s)\right)  ^{2}\right)  \left\|  f(s)-h(s)\right\|  _{\alpha}^{2}%
\mathrm{d}s.
\end{align*}

\end{proof}

Finally, we will consider the It\^{o} stochastic integral with respect to the
$r$-dimensional standard Brownian motion $W$. The following lemma is an
immediate consequence of It\^{o} calculus. \ 

\begin{lemma}
Suppose that $u=\{u(t),t\in\lbrack0,T]\}$ is an $r$-dimensional $\mathcal{G}%
_{t}$-adapted stochastic process such that $\ \int_{0}^{T}\mathbb{E}%
^{W}\left[  u(s)^{2}\right]  \mathrm{d}s<\infty$. Then for all $t\in
\lbrack0,T]$ a.e.%
\begin{equation}
\mathbb{E}^{W}\left[  \left\|  \int_{0}^{t}u(s)\mathrm{d}W_{s}\right\|
_{\alpha}^{2}\right]  \leq C\int_{0}^{t}(t-s)^{-\frac{1}{2}-\alpha}%
\mathbb{E}^{W}\left[  u(s)^{2}\right]  \mathrm{d}s. \label{GWf}%
\end{equation}

\end{lemma}

\begin{proof}
Notice first that, by Fubini's theorem, the right-hand side of (\ref{GWf}) is
finite for all $t\in\lbrack0,T]$ a.e. Applying the It\^{o} isometry property
and the Cauchy-Schwarz inequality, we have that:
\begin{align*}
&  \mathbb{E}^{W}\left[  \left\|  \int_{0}^{t}u(s)\mathrm{d}W_{s}\right\|
_{\alpha}^{2}\right]  \leq C\mathbb{E}^{W}\left[  \left|  \int_{0}%
^{t}u(s)\mathrm{d}W_{s}\right|  ^{2}+\left(  \int_{0}^{t}\frac{\left|
\int_{s}^{t}u(r)\mathrm{d}W_{r}\right|  }{\left(  t-s\right)  ^{\alpha+1}%
}\mathrm{d}s\right)  ^{2}\right] \\
&  \leq C\left[  \int_{0}^{t}\mathbb{E}^{W}\left[  u(s)^{2}\right]
\mathrm{d}s+\frac{T^{\frac{1}{2}-\alpha}}{\frac{1}{2}-\alpha}\int_{0}^{t}%
\frac{\mathbb{E}^{W}\left|  \int_{s}^{t}u(r)\mathrm{d}W_{r}\right|  ^{2}%
}{\left(  t-s\right)  ^{\frac{3}{2}+\alpha}}\mathrm{d}s\right]  .
\end{align*}
Therefore, by the It\^{o} isometry and Fubini's theorem we obtain
\begin{align*}
&  \mathbb{E}^{W}\left[  \left\|  \int_{0}^{t}u(s)\mathrm{d}W_{s}\right\|
_{\alpha}^{2}\right] \\
&  \leq C\left[  \int_{0}^{t}\mathbb{E}^{W}\left[  u(s)^{2}\right]
\mathrm{d}s+\frac{T^{\frac{1}{2}-\alpha}}{\left(  \frac{1}{2}-\alpha\right)
}\int_{0}^{t}\left[  \int_{0}^{r}\left(  t-s\right)  ^{-\frac{3}{2}-\alpha
}\mathrm{d}s\right]  \mathbb{E}^{W}\left[  u(r)^{2}\right]  \mathrm{d}r\right]
\\
&  \leq C\left[  \int_{0}^{t}\mathbb{E}^{W}\left[  u(s)^{2}\right]
\mathrm{d}s+\frac{T^{\frac{1}{2}-\alpha}}{\left(  \frac{1}{2}-\alpha\right)
\left(  \frac{1}{2}+\alpha\right)  }\int_{0}^{t}\frac{\mathbb{E}^{W}\left[
u(r)^{2}\right]  }{\left(  t-r\right)  ^{\frac{1}{2}+\alpha}}\mathrm{d}%
r\right] \\
&  \leq C\int_{0}^{t}\left(  t-r\right)  ^{-\frac{1}{2}-\alpha}\mathbb{E}%
^{W}\left[  u(r)^{2}\right]  \mathrm{d}r.
\end{align*}

\end{proof}

Assume that $f=\{f(t),t\in\lbrack0,T]\}$ is a $d$-dimensional stochastic
process in $\mathbb{W}_{\mathcal{G}}$. Define%

\[
G_{t}^{\sigma W}(f):=\int_{0}^{t}\sigma(s,f(s))\mathrm{d}W_{s}.
\]
We have the following estimates for these integrals:

\begin{proposition}
Let $\ f\in\mathbb{W}_{\mathcal{G}}$. Then for all $t\in\lbrack0,T]$ a.e.
\begin{equation}
\mathbb{E}^{W}\left[  \left\|  G_{t}^{\sigma W}(f)\right\|  _{\alpha}%
^{2}\right]  \leq C\int_{0}^{t}(t-s)^{-\frac{1}{2}-\alpha}\left[
1+\mathbb{E}^{W}\left[  \left\|  f(s)\right\|  _{\alpha}^{2}\right]  \right]
\mathrm{d}s. \label{GsigmaWf2}%
\end{equation}

\end{proposition}

\begin{proof}
It follows from (\ref{GWf}) and the linear growth assumption in $\ $%
(H$\sigma_{W}$).
\end{proof}

\begin{proposition}
Let $\ f,h\in$ $\mathbb{W}_{\mathcal{G}}$. Then for all $t\in\lbrack0,T]$
a.e.
\begin{equation}
\mathbb{E}^{W}\left[  \left\|  G_{t}^{\sigma W}(f)-G_{t}^{\sigma
W}(h)\right\|  _{\alpha}^{2}\right]  \leq C\int_{0}^{t}(t-s)^{-\frac{1}%
{2}-\alpha}\mathbb{E}^{W}\left[  \left|  f(s)-h(s)\right|  ^{2}\right]
\mathrm{d}s. \label{GW2}%
\end{equation}

\end{proposition}

\begin{proof}
By estimate (\ref{GWf}) and the Lipschitz assumption in (H$\sigma_{W}$), we
obtain
\begin{align*}
&  \mathbb{E}^{W}\left[  \left\Vert G_{t}^{\sigma W}(f)-G_{t}^{\sigma
W}(h)\right\Vert _{\alpha}^{2}\right] \\
&  \leq C\int_{0}^{t}(t-s)^{-\frac{1}{2}-\alpha}\mathbb{E}^{W}\left[
\left\vert \sigma_{W}(s,f(s))-\sigma_{W}(s,h(s))\right\vert ^{2}\right]
\mathrm{d}s\\
&  \leq C\int_{0}^{t}(t-s)^{-\frac{1}{2}-\alpha}\mathbb{E}^{W}\left[
\left\vert f(s)-h(s)\right\vert ^{2}\right]  \mathrm{d}s.
\end{align*}

\end{proof}

\setcounter{equation}{0}

\section{Pathwise uniqueness}

In this section we define the notion of weak solution for the stochastic
differential equation (\ref{sde1}) and we discuss the pathwise uniqueness of a solution.

\begin{definition}
\label{Def: Weak Solution}A weak solution of the stochastic differential
equation (\ref{sde1}) is a triple $\left(  X,B^{H},W\right)  ,$ $(\Omega
,\mathcal{F},P),$ $\left\{  \mathcal{G}_{t},t\in\lbrack0,T]\right\}  ,$ where

\begin{enumerate}
\item $(\Omega,\mathcal{F},P)$ is a complete probability space, $\left\{
\mathcal{G}_{t}\right\}  $ is a \ right-continuous filtration such that
$\mathcal{G}_{0}$ contains the $P$-\ null sets.

\item $W$ is a $\mathcal{G}_{t}$-$r$-dimensional Brownian motion.

\item $B^{H}$ is a fractional Brownian motion of Hurst parameter $H$ which is
$\mathcal{G}_{0}$-measurable.

\item The process $X$ is $\mathcal{G}_{t}$-adapted, has trajectories in
$W_{0}^{\alpha,\infty}$ almost surely, and $\int_{0}^{T}\mathbb{E}^{W}\left[
\left\|  X_{s}\right\|  _{\alpha}^{2}\right]  \mathrm{d}s<\infty$ a.s.

\item $\left(  X,B^{H},W\right)  $ satisfies Equation (\ref{sde1}) a.s.
\end{enumerate}
\end{definition}

\begin{definition}
We say that pathwise uniqueness holds for Equation (\ref{sde1}) if, whenever
$\left(  X,W,B^{H}\right)  $ and $\left(  Y,W,B^{H}\right)  $\ are two weak
solutions, defined on the same probability space $(\Omega,\mathcal{F},P)$ with
the same filtration $\left\{  \mathcal{G}_{t}\right\}  $ and $X_{0}=Y_{0}$
a.s., then $\ X=Y$.
\end{definition}

$\ $We will make use of the following technical lemma.

\begin{lemma}
\label{LemmaH} Let $0<\eta<1/2$. If $f$ is a continuous function such that
$\left\Vert f\right\Vert _{\eta}\leq N$ and $\alpha<\eta\delta$, then
$\Delta(f)$ is bounded by a constant $C$ depending on $T$, $N$, $\alpha$,
$\delta$, and $\eta$, where we use the notation introduced in (\ref{Hnorm})
and (\ref{Delta}).
\end{lemma}

\begin{proof}
Clearly%
\[
\Delta(f)(s)=\int_{0}^{s}\frac{\left\vert f(s)-f(r)\right\vert ^{\delta}%
}{(s-r)^{\alpha+1}}\mathrm{d}r\leq N^{\delta}\frac{T^{\eta\delta-\alpha}}%
{\eta\delta-\alpha},
\]
which gives the result.
\end{proof}

Let $f\in W_{0}^{\alpha,\infty}$. By the estimates proved in Proposition 4.2
and Proposition 4.4 of \cite{Nualart}, the sample paths of the\ integral
processes $F^{b}(f)$ and $G^{\sigma H}(f)$ are continuously differentiable and
$\eta$-H\"{o}lder continuous of order $\eta<1-\alpha$, respectively.
Therefore, if $X$ is a weak solution of (\ref{sde1}), then the trajectories of
$X$ are $\eta$-H\"{o}lder continuous for all $\eta<1/2$.

\begin{theorem}
\label{Uniqueness}\textbf{(Pathwise uniqueness)} Let $1-H<\alpha<\min\left\{
\beta,\frac{\delta}{2},\frac{1}{2}\right\}  $. Then, the pathwise uniqueness
property holds for Equation (\ref{sde1}).
\end{theorem}

\begin{proof}
Let $X$ and $Y$ be two weak solutions of (\ref{sde1}) defined on the same
probability space, adapted to the same filtration and with the same initial
value. Then the trajectories of $X$ and $Y$ are $\eta$-H\"{o}lder continuous,
for all $\eta<1/2$. Choose $\eta$ such that $\alpha<\eta<1/2$. Consider the
sets $\Omega_{N}\subset\Omega$, defined by
\[
\Omega_{N}:=\left\{  \omega\in\Omega:\text{ }\left\Vert X\right\Vert _{\eta
}\leq N\text{ and }\left\Vert Y\right\Vert _{\eta}\leq N\text{ }\right\}  ,
\]
with $N\in\mathbb{N}$. It is clear that $\Omega_{N}\nearrow\Omega$. From
(\ref{sde1}) we have that the difference between the two solutions satisfies
\begin{align}
&  \mathbb{E}^{W}\left[  \left\Vert X_{t}-Y_{t}\right\Vert _{\alpha}%
^{2}\mathbf{1}_{\Omega_{N}}\right] \nonumber\\
&  \leq4\mathbb{E}^{W}\left\Vert \left(  F_{t}^{b}(X)-F_{t}^{b}(Y)\right)
\mathbf{1}_{\Omega_{N}}\right\Vert _{\alpha}^{2}+4\mathbb{E}^{W}\left\Vert
\left(  G_{t}^{\sigma W}(X)-G_{t}^{\sigma W}(Y)\right)  \right\Vert _{\alpha
}^{2}\nonumber\\
&  +4\mathbb{E}^{W}\left\Vert \left(  G_{t}^{\sigma H}(X)-G_{t}^{\sigma
H}(Y)\right)  \mathbf{1}_{\Omega_{N}}\right\Vert _{\alpha}^{2}.
\label{Eq difference solutions}%
\end{align}
We split the set $\Omega$ into $\Omega_{N}$ and $\Omega\backslash\Omega_{N}$
in the second summand of (\ref{Eq difference solutions}) and use the estimates
(\ref{Fbfh}), (\ref{GHfh}), (\ref{GW2}) in order to obtain%
\begin{align}
&  \mathbb{E}^{W}\left[  \left\Vert X_{t}-Y_{t}\right\Vert _{\alpha}%
^{2}\mathbf{1}_{\Omega_{N}}\right] \nonumber\\
&  \leq C\int_{0}^{t}\varphi(s,t)\mathbb{E}^{W}\left[  \left\Vert X_{s}%
-Y_{s}\right\Vert _{\alpha}^{2}\mathbf{1}_{\Omega_{N}}\right]  \mathrm{d}%
s\nonumber\\
&  +C\int_{0}^{t}\varphi(s,t)\left(  \mathbb{E}^{W}\left[  \left\Vert
X_{s}-Y_{s}\right\Vert _{\alpha}^{2}\mathbf{1}_{\Omega_{N}}\right]
+\mathbb{E}^{W}\left[  \left\Vert X_{s}-Y_{s}\right\Vert _{\alpha}%
^{2}\mathbf{1}_{\Omega\backslash\Omega_{N}}\right]  \right)  \mathrm{d}%
s\nonumber\\
&  +C\left(  \Lambda_{\alpha}(B^{H})\right)  ^{2}\int_{0}^{t}\varphi
(s,t)\mathbb{E}^{W}\left[  \left(  1+(\Delta X_{s})^{2}+(\Delta Y_{s}%
)^{2}\right)  \left\Vert X_{s}-Y_{s}\right\Vert _{\alpha}^{2}\mathbf{1}%
_{\Omega_{N}}\right]  \mathrm{d}s, \label{Eq uniqueness estimate}%
\end{align}
where%
\[
\varphi(s,t)=(t-s)^{-\frac{1}{2}-\alpha}+s^{-\alpha}\text{. }%
\]
If $\omega\in\Omega_{N}$ then, by Lemma \ref{LemmaH}, we have that
\begin{equation}
1+(\Delta X_{s})^{2}+(\Delta Y_{s})^{2}\leq C_{N}\text{.} \label{z2}%
\end{equation}
Set%
\[
V_{N}(t)=\int_{0}^{t}\varphi(s,t)\mathbb{E}^{W}\left[  \left\Vert X_{s}%
-Y_{s}\right\Vert _{\alpha}^{2}\mathbf{1}_{\Omega_{N}}\right]  \mathrm{d}s.
\]
Multiplying Equation (\ref{Eq uniqueness estimate}) by $\varphi(s,t)$ and
integrating, yields%
\begin{align}
V_{N}(t)  &  \leq C_{N}\left[  \left(  \Lambda_{\alpha}(B^{H})\right)
^{2}+1\right]  \int_{0}^{t}\varphi(s,t)V_{N}(s)\mathrm{d}s\label{z3b}\\
&  +C\int_{0}^{t}\varphi(s,t)\int_{0}^{s}\varphi(r,s)\mathbb{E}^{W}\left[
\left\Vert X_{r}-Y_{r}\right\Vert _{\alpha}^{2}\mathbf{1}_{\Omega
\backslash\Omega_{N}}\right]  \mathrm{d}r\mathrm{d}s.\nonumber
\end{align}
By the bounded convergence theorem, we have that almost surely%
\[
V_{N}(t)\ \rightarrow\int_{0}^{t}\varphi(s,t)\mathbb{E}^{W}\left[  \left\Vert
X_{s}-Y_{s}\right\Vert _{\alpha}^{2}\right]  \mathrm{d}s<\infty
\]
and%
\[
\int_{0}^{t}\varphi(s,t)\int_{0}^{s}\varphi(r,s)\mathbb{E}^{W}\left[
\left\Vert X_{r}-Y_{r}\right\Vert _{\alpha}^{2}\mathbf{1}_{\Omega
\backslash\Omega_{N}}\right]  \mathrm{d}r\mathrm{d}s\rightarrow0,
\]
as $N$ tends to infinity. Then, there exists a random variable $N^{\ast}%
\in\mathbb{N}$ such that
\begin{equation}
C\int_{0}^{t}\varphi(s,t)\int_{0}^{s}\varphi(r,s)\mathbb{E}^{W}\left[
\left\Vert X_{r}-Y_{r}\right\Vert _{\alpha}^{2}\mathbf{1}_{\Omega
\backslash\Omega_{N}}\right]  \mathrm{d}r\mathrm{d}s\leq\frac{1}{2}V_{N}(t),
\label{z4}%
\end{equation}
for all $N\geq N^{\ast}$. Substituting (\ref{z4}) into (\ref{z3b}) yields
\[
\ V_{N}(t)\leq C_{N}\left[  \left(  \Lambda_{\alpha}(B^{H})\right)
^{2}+1\right]  \int_{0}^{t}\varphi(s,t)V_{N}(s)\mathrm{d}s,
\]
for all $N\geq N^{\ast}$. Applying now the Gronwall-type Lemma 7.6 in
\cite{Nualart}, we deduce that $V_{N}(t)=0$ for all $N\geq N^{\ast}$ almost
surely. Hence, $\ $
\[
P\left[  X_{t}=Y_{t},\text{ \ }\forall t\in\left[  0,T\right]  \right]  =1,
\]
and the pathwise uniqueness property holds.
\end{proof}

\setcounter{equation}{0}

\section{Existence of solutions}

Let us now introduce the Euler approximations for Equation (\ref{sde1}).
Consider the framework $(\Omega,\mathcal{F},P)$, $\{\mathcal{F}_{t}%
,t\in\lbrack0,T]\}$, $(X_{0},B^{H},W)$ introduced in Section 2. Fix a sequence
of partitions
\[
0=t_{0}^{n}<t_{1}^{n}<\cdots<t_{i}^{n}<\cdots<t_{n}^{n}=T
\]
of $\left[  0,T\right]  $ such that
\[
\underset{0\leq i\leq n-1}{\sup}\left|  t_{i+1}^{n}-t_{i}^{n}\right|
\rightarrow0
\]
as $n\rightarrow\infty$. Define $X^{0}(t)=X_{0}$ and for $n\geq1,$%

\begin{align}
X^{n}(t)  &  =X_{0}+\int_{0}^{t}b(k_{n}(s),X^{n}(k_{n}(s)))\mathrm{d}%
s+\int_{0}^{t}\sigma_{W}(k_{n}(s),X^{n}(k_{n}(s)))\mathrm{d}W_{s}\nonumber\\
&  +\int_{0}^{t}\sigma_{H}(k_{n}(s),X^{n}(k_{n}(s)))\mathrm{d}B_{s}^{H},
\label{eqEuler}%
\end{align}
where
\[
k_{n}(t):=t_{i}^{n},
\]
if $t\in\lbrack t_{i}^{n},t_{i+1}^{n})$. We will show the following result.

\begin{proposition}
\label{prop1} For any integer $N\geq1$ there exists a random variable
$R_{N}>0$, depending on $X_{0}$ and $B^{H}$, such that, almost surely,
\begin{equation}
\mathbb{E}^{W}\left[  \left|  X_{t}^{n}-X_{s}^{n}\right|  ^{2N}\right]  \leq
R_{N}\ \left|  t-s\right|  ^{N}, \label{z7}%
\end{equation}
for all $s,t\in\lbrack0,T]$ and $n\in\mathbb{N}$.
\end{proposition}

\begin{proof}
The proof will be done in two steps.

\textit{Step 1.-} We begin by proving that there is a random variable
$K_{N}>0$ such that
\begin{equation}
\mathbb{E}^{W}\left[  \left\|  X_{t}^{n}\right\|  _{\alpha}^{2N}\right]  \leq
K_{N}, \label{z6}%
\end{equation}
for all $t\in\left[  0,T\right]  $ and for all $N\in\mathbb{N}$.

Note that the paths of $X^{n}(k_{n}(\cdot))$ are piecewise constant and the
integrals in (\ref{eqEuler}) are just finite sums. In the following
computations, $C_{N}\ $denotes a positive constant that depends on $N$ and the
other parameters of the problem, and may vary from line to line. From
(\ref{eqEuler}), we have that
\begin{align*}
&  \mathbb{E}^{W}\left[  \left\|  X_{t}^{n}\right\|  _{\alpha}^{2N}\right]
\leq C_{N}\left\{  \left|  X_{0}\right|  ^{2N}+\mathbb{E}^{W}\left[  \left\|
\int_{0}^{t}b(k_{n}(s),X^{n}(k_{n}(s)))\mathrm{d}s\right\|  _{\alpha}%
^{2N}\right]  \right. \\
&  +\mathbb{E}^{W}\left[  \left\|  \int_{0}^{t}\sigma_{W}(k_{n}(s),X^{n}%
(k_{n}(s)))\mathrm{d}W_{s}\right\|  _{\alpha}^{2N}\right] \\
&  \left.  +\mathbb{E}^{W}\left[  \left\|  \int_{0}^{t}\sigma_{H}%
(k_{n}(s),X^{n}(k_{n}(s)))\mathrm{d}B_{s}^{H}\right\|  _{\alpha}^{2N}\right]
\right\} \\
&  =C_{N}\left(  \left|  X_{0}\right|  ^{2N}+A_{1}+A_{2}+A_{3}\right)  .
\end{align*}
Using the estimate (\ref{Fbf}) and H\"{o}lder's inequality, we obtain
\begin{align*}
&  A_{1}\leq C_{N}\mathbb{E}^{W}\left[  \left(  \int_{0}^{t}\frac{\left|
X^{n}(k_{n}(s))\right|  }{(t-s)^{\alpha}}\mathrm{d}s+1\right)  ^{2N}\right] \\
&  \leq C_{N}\mathbb{E}^{W}\left[  \left(  \int_{0}^{t}\left|  X^{n}%
(k_{n}(s))\right|  ^{2}\mathrm{d}s\right)  ^{N}\right]  +C_{N}\\
&  \leq C_{N}\mathbb{E}^{W}\left[  \int_{0}^{t}\left|  X^{n}(k_{n}(s))\right|
^{2N}\mathrm{d}s\right]  +C_{N}.
\end{align*}
We have also that
\begin{align*}
&  A_{2}\leq C_{N}\mathbb{E}^{W}\left[  \left|  \int_{0}^{t}\sigma_{W}%
(k_{n}(s),X^{n}(k_{n}(s)))\mathrm{d}W_{s}\right|  ^{2N}\right] \\
&  +C_{N}\mathbb{E}^{W}\left[  \left(  \int_{0}^{t}\frac{\left|  \int_{s}%
^{t}\sigma_{W}(k_{n}(r),X^{n}(k_{n}(r)))\mathrm{d}W_{r}\right|  }{\left(
t-s\right)  ^{\alpha+1}}\mathrm{d}s\right)  ^{2N}\right] \\
&  =A_{11}+A_{12}.
\end{align*}
Applying Burkh\"{o}lder and H\"{o}lder inequalities, we have that
\begin{align*}
&  A_{11}\leq C_{N}\mathbb{E}^{W}\left[  \left|  \int_{0}^{t}\left|
\sigma_{W}(k_{n}(s),X^{n}(k_{n}(s)))\right|  ^{2N}\mathrm{d}s\right|  \right]
\\
&  \leq C_{N}\int_{0}^{t}\left(  1+\mathbb{E}^{W}\left[  \left|  X^{n}%
(k_{n}(s))\right|  ^{2N}\right]  \right)  \mathrm{d}s,
\end{align*}
where we have used the linear growth assumption in $\ $(H$\sigma_{W}$). For
the second term we have, by H\"{o}lder and Burkh\"{o}lder inequalities, that
\begin{align*}
A_{12}  &  \leq C_{N}\mathbb{E}^{W}\left[  \left(  \int_{0}^{t}\frac
{1}{\left(  t-s\right)  ^{\frac{2N}{2N-1}\left(  \alpha+\frac{1}{2}%
-\frac{1/2+\alpha}{2N}\right)  }}\mathrm{d}s\right)  ^{2N-1}\right. \\
&  \left.  \times\left(  \int_{0}^{t}\frac{\left|  \int_{s}^{t}\sigma
_{W}(k_{n}(r),X^{n}(k_{n}(r))\mathrm{d}W_{r}\right|  ^{2N}}{\left(
t-s\right)  ^{N+\frac{1}{2}+\alpha}}\mathrm{d}s\right)  \right] \\
&  \leq C_{N}\int_{0}^{t}(t-s)^{-\frac{3}{2}-\alpha}\mathbb{E}^{W}\left[
\int_{s}^{t}\left|  \sigma_{W}(k_{n}(r),X^{n}(k_{n}(r)))\right|
^{2N}\mathrm{d}r\right]  \mathrm{d}s.
\end{align*}
Applying now Fubini's theorem and using the growth assumption in (H$\sigma
_{W}$), we obtain
\[
A_{12}\leq C_{N}\int_{0}^{t}\left(  t-r\right)  ^{-\frac{1}{2}-\alpha}\left(
1+\mathbb{E}^{W}\left[  \left|  X^{n}(k_{n}(r))\right|  ^{2N}\right]  \right)
\mathrm{d}r.
\]
Therefore,
\[
A_{2}\leq C_{N}\int_{0}^{t}\left(  t-s\right)  ^{-\frac{1}{2}-\alpha
}\mathbb{E}^{W}\left[  \left|  X^{n}(k_{n}(s))\right|  ^{2N}\right]
\mathrm{d}s+C_{N}.
\]
Applying (\ref{GsigmaHf2}), we have that
\[
A_{3}\leq C_{N}\Lambda_{\alpha}(B^{H})^{2N}\ \mathbb{E}^{W}\left[  \left(
\int_{0}^{t}\left(  (t-s)^{-2\alpha}+s^{-\alpha}\right)  \left\|  \sigma
_{H}(k_{n}(s),X^{n}(k_{n}(s)))\right\|  _{\alpha}\mathrm{d}s\right)
^{2N}\right]  .
\]
By H\"{o}lder's inequality and the assumptions in (H$\sigma_{H}$), we have
\[
A_{3}\leq C_{N}\Lambda_{\alpha}(B^{H})^{2N}\int_{0}^{t}\left(  (t-s)^{-2\alpha
}+s^{-\alpha}\right)  \left[  1+\mathbb{E}^{W}\left[  \left\|  X^{n}%
(k_{n}(s))\right\|  _{\alpha}^{2N}\right]  \right]  \mathrm{d}r.
\]
Putting together all the estimates obtained for $A_{1},A_{2}$ and $A_{3},$ we
obtain
\begin{align}
&  \mathbb{E}^{W}\left[  \left\|  X_{t}^{n}\right\|  _{\alpha}^{2N}\right]
\leq C_{N}\left|  X_{0}\right|  ^{2N}+C_{N}\left[  \Lambda_{\alpha}%
(B^{H})^{2N}+1\right] \nonumber\\
&  \quad\times\int_{0}^{t}\left(  \left(  t-s\right)  ^{-\frac{1}{2}-\alpha
}+s^{-\alpha}\right)  \mathbb{E}^{W}\left[  \left\|  X^{n}(k_{n}(s))\right\|
_{\alpha}^{2N}\right]  \mathrm{d}s. \label{z5}%
\end{align}
Therefore, since the right-hand side of Equation (\ref{z5}) is an increasing
function of $t$, we have
\begin{align*}
&  \underset{0\leq s\leq t}{\sup}\mathbb{E}^{W}\left[  \left\|  X_{s}%
^{n}\right\|  _{\alpha}^{2N}\right]  \leq C_{N}\left|  X_{0}\right|
^{2N}+C_{N}\left[  \Lambda_{\alpha}(B^{H})^{2N}+1\right] \\
&  \quad\times\int_{0}^{t}\left(  \left(  t-s\right)  ^{-\frac{1}{2}-\alpha
}+s^{-\alpha}\right)  \left(  \underset{0\leq u\leq s}{\sup}\mathbb{E}%
^{W}\left[  \left\|  X_{u}^{n}\right\|  _{\alpha}^{2N}\right]  \right)
\mathrm{d}s.
\end{align*}
As a consequence, by the Gronwall-type lemma (Lemma 7.6 in \cite{Nualart}), we
deduce the desired estimate.

\textit{Step 2.-} Now we show that there is a random variable $R_{N}$ such
that (\ref{z7}) holds. In the sequel, $R_{N}$ denotes a positive random
variable. We have
\begin{align*}
&  \mathbb{E}^{W}\left[  \left|  X_{t}^{n}-X_{s}^{n}\right|  ^{2N}\right]
\leq C_{N}\left\{  \mathbb{E}^{W}\left[  \left|  \int_{s}^{t}b(k_{n}%
(u),X^{n}(k_{n}(u)))\mathrm{d}u\right|  ^{2N}\right]  \right. \\
&  +\mathbb{E}^{W}\left[  \left|  \int_{s}^{t}\sigma_{W}(k_{n}(u),X^{n}%
(k_{n}(u)))\mathrm{d}W_{u}\right|  ^{2N}\right] \\
&  \left.  +\mathbb{E}^{W}\left[  \left|  \int_{s}^{t}\sigma_{H}%
(k_{n}(u),X^{n}(k_{n}(u)))\mathrm{d}B_{u}^{H}\right|  ^{2N}\right]  \right\}
\\
&  =B_{1}+B_{2}+B_{3}.
\end{align*}
Applying H\"{o}lder's inequality, the growth assumption in (H$b$) and
(\ref{z6}), we have that%
\begin{align*}
&  B_{1}\leq C_{N}(t-s)^{2N-1}\int_{s}^{t}\mathbb{E}^{W}\left[  \left|
b(k_{n}(u),X^{n}(k_{n}(u))\right|  ^{2N}\right]  \mathrm{d}u\\
&  \leq R_{N}(t-s)^{2N}.
\end{align*}
By the H\"{o}lder and Burkh\"{o}lder inequalities and using (\ref{z6}), we
obtain
\[
B_{2}\leq C_{N}(t-s)^{N-1}\mathbb{E}^{W}\left[  \int_{s}^{t}\left|  \sigma
_{W}(k_{n}(u),X^{n}(k_{n}(u))\right|  ^{2N}\mathrm{d}u\right]  \leq
R_{N}(t-s)^{N}.
\]
Finally, using the estimate (\ref{Gfa1}) and the H\"{o}lder inequality, we
have
\[
\left|  \int_{s}^{t}f(u)\mathrm{d}B_{u}^{H}\right|  ^{2N}\leq C_{N}%
\Lambda_{\alpha}(B^{H})^{2N}\left(  t-s\right)  ^{2N(1-\alpha)+2\alpha-1}%
\int_{s}^{t}\frac{\left\|  f(r)\right\|  _{\alpha}^{2N}}{\left(  r-s\right)
^{2\alpha}}\mathrm{d}r.
\]
Applying this estimate, the assumptions (H$\sigma_{H}$) and (\ref{z6}), we
obtain
\begin{align*}
&  B_{3}\leq\mathbb{E}^{W}\Lambda_{\alpha}(B^{H})^{2N}\left(  t-s\right)
^{2N(1-\alpha)+2\alpha-1}\int_{s}^{t}\frac{\left\|  \sigma_{H}(k_{n}%
(r),X^{n}(k_{n}(r)))\right\|  _{\alpha}^{2N}}{\left(  r-s\right)  ^{2\alpha}%
}\mathrm{d}r\\
&  \leq R_{N}\left(  t-s\right)  ^{2N(1-\alpha)+2\alpha-1}\mathbb{E}%
^{W}\left[  \int_{s}^{t}\frac{1+\left\|  X^{n}(k_{n}(r))\right\|  _{\alpha
}^{2N}}{\left(  r-s\right)  ^{2\alpha}}\mathrm{d}r\right] \\
&  \leq R_{N}(t-s)^{N},
\end{align*}
which concludes the proof.
\end{proof}

As a consequence of Proposition \ref{prop1}, we establish the tightness of the
law of the sequence $\left\{  X^{n}\right\}  _{n\in\mathbb{N}}$ in\ the space
$C_{0}^{\eta}$ of $\eta$-H\"{o}lder continuous functions, with $\eta<\frac
{1}{2}$, such that
\[
\underset{\varepsilon\rightarrow0}{\lim}\text{ }\underset{0<\left\vert
t-s\right\vert <\varepsilon}{\sup}\frac{\left\vert f(t)-f(s)\right\vert
}{(t-s)^{\eta}}=0.
\]
These spaces are complete and separable \cite{Hamadouche}.

\begin{proposition}
Let $P^{n}=P\circ X^{n}$, $n\geq0$, be the sequence of probability measures
induced by $X^{n}$ on $C_{0}^{\eta}$. Then this sequence is tight.
\end{proposition}

\begin{proof}
Fix $\varepsilon>0$ and $\eta<\frac{1}{2}$. It suffices to show that there
exists a compact set $K$ in $C_{0}^{\eta}$ such that $\sup_{n\geq0}P(X^{n}\in
K^{c})\leq\varepsilon$. Choose an integer $N$ such that $\frac{1}{2}-\frac
{1}{2N}>\eta$. Let $M>0$ be such that%
\begin{equation}
P(R_{N}>M)\leq\frac{\varepsilon}{2}\text{. } \label{xx1}%
\end{equation}
Define a new probability by%
\[
Q(B)=\frac{P(B\cap\left\{  R_{N}\leq M\right\}  )}{P(R_{N}\leq M)}.
\]
Then, Proposition \ref{prop1} implies that%
\[
\mathbb{E}_{Q}\left[  \left|  X_{t}^{n}-X_{s}^{n}\right|  ^{2N}\right]
=\frac{\mathbb{E}\left[  \left|  X_{t}^{n}-X_{s}^{n}\right|  ^{2N}%
\mathbf{1}_{\left\{  R_{N}\leq M\right\}  }\right]  }{P(R_{N}\leq M)}\leq
MP(R_{N}\leq M)^{-1}\ \left|  t-s\right|  ^{N}.
\]
By the tightness criterion established in \cite{Lamperti}, the sequence \ of
probabilities $Q\circ X_{n}^{-1}$, $n\geq0$, is tight in $C_{0}^{\eta}$.
Therefore, there exists a compact subset $K$ in $C_{0}^{\eta}$ such that%
\begin{equation}
\sup_{n\geq0}Q(X^{n}\in K^{c})\leq P(R_{N}\leq M)^{-1}\frac{\varepsilon}{2}.
\label{xx2}%
\end{equation}
Finally, from (\ref{xx1}) and (\ref{xx2}) we obtain%
\[
P(X^{n}\in K^{c})\leq P(X^{n}\in K^{c},R_{N}\leq M)+P(R_{N}>M)\leq
\varepsilon,
\]
which allows us to conclude the proof.
\end{proof}

Now we can show the existence of a weak solution for Equation (\ref{sde1}).

\begin{theorem}
Assume that the coefficients $b,$ $\sigma_{W}$ and $\sigma_{H}$ satisfy the
assumptions (H$b$), (H$\sigma_{W}$) and (H$\sigma_{H}$). If $1-H<\alpha
<\min\left\{  \frac{1}{2},\beta,\frac{\delta}{2}\right\}  $, then there exists
a unique weak solution\ $X$ of Equation (\ref{sde1}).
\end{theorem}

\begin{proof}
The proof will be done in several steps.

\textit{Step 1.- }By the Prohorov theorem, the sequence $\{P^{n},n\geq0\}$ is
weakly relatively compact in $C_{0}^{\eta}$ and exists a subsequence,\ that we
still denote by $P^{n},$ which is weakly convergent to some probability
$P^{\infty}$. By the Skorokhod representation theorem, there exists a sequence
of processes $\left\{  Y^{n},B^{n},W^{n},0\leq n\leq\infty\right\}  ,$ defined
on some probability space $(\Omega,\mathcal{F},P)$ and with values in
$C_{0}^{\eta}$, such that, for every $0\leq n\leq\infty$, the process $\left(
Y^{n},B^{n},W^{n}\right)  $ has law $P^{n}$ and
\[
\lim_{n\longrightarrow\infty}\left\Vert Y^{n}-Y^{\infty}\right\Vert _{\eta
}+\left\Vert B^{n}-B^{\infty}\right\Vert _{\eta}+\left\Vert W^{n}-W^{\infty
}\right\Vert _{\eta}=0
\]
almost surely.

Since, for every $n$, the process $(Y^{n},B^{n},W^{n})$ has the same law as
$(X^{n},B^{H},W)$, if we introduce the filtrations%
\begin{align*}
\mathcal{F}_{t}^{n}  &  =\sigma\left\{  Y^{n}(s),B^{n}(s),W^{n}(s),\text{
}s\leq t\right\}  ,\\
\mathcal{F}_{t}^{\infty}  &  =\sigma\left\{  Y^{\infty}(s),B^{\infty
}(s),W^{\infty}(s),\text{ }s\leq t\right\}  ,
\end{align*}
the process $W^{n}$ (resp. $W^{\infty}$) is an $\mathcal{F}_{t}^{n}$ (resp.
$\mathcal{F}_{t}^{\infty}$) $r$-dimensional standard Brownian motion.
Moreover, $B^{n}$ and $B^{\infty}$ are fractional Brownian motions.

\textit{Step 2.- }By an adaptation of a result in \cite{Skorokhod} (page 32)
or Lemma 3.1 in \cite{Gyongy}, for any continuous function $f(t,x)$ which
satisfies the linear growth property in the variable $x$, we have that%
\begin{align*}
\lim_{n\longrightarrow\infty}\int_{0}^{t}f(k_{n}(s),Y^{n}(k_{n}(s)))\mathrm{d}%
s  &  =\int_{0}^{t}f(s,Y^{\infty}(s))\mathrm{d}s,\\
\lim_{n\longrightarrow\infty}\int_{0}^{t}f(k_{n}(s),Y^{n}\left(
k_{n}(s)\right)  )\mathrm{d}W_{s}^{n}  &  =\int_{0}^{t}f(s,Y^{\infty
}(s))\mathrm{d}W_{s}^{\infty},
\end{align*}
as $n$ tends to infinity, in probability, and uniformly in $t\in\left[
0,T\right]  $. We have also a similar result for the convergence of integrals
with respect to fractional Brownian motions:%
\begin{equation}
\lim_{n\longrightarrow\infty}\int_{0}^{t}\sigma_{H}(k_{n}\left(  s\right)
,Y^{n}(k_{n}\left(  s\right)  ))\mathrm{d}B_{s}^{n}=\int_{0}^{t}\sigma
_{H}(s,Y^{\infty}(s))\mathrm{d}B_{s}^{\infty}, \label{x1}%
\end{equation}
as $n$ tends to infinity, uniformly in $t\in\left[  0,T\right]  $ and $P$-a.s.
Let us show the convergence (\ref{x1}). By the linearity of the generalized
Stieltjes integral, it is clear that%
\[
\left\vert \int_{0}^{t}\sigma_{H}(k_{n}\left(  s\right)  ,Y^{n}(k_{n}\left(
s\right)  ))\mathrm{d}B_{s}^{n}-\int_{0}^{t}\sigma_{H}(s,Y^{\infty
}(s))\mathrm{d}B_{s}^{\infty}\right\vert \leq A_{1}+A_{2}\text{,}%
\]
where%
\[
A_{1}=\left\vert \int_{0}^{t}\sigma_{H}(k_{n}\left(  s\right)  ,Y^{n}%
(k_{n}\left(  s\right)  ))\mathrm{d}(B_{s}^{n}-B_{s}^{\infty})\right\vert
\]
and%
\[
A_{2}=\left\vert \int_{0}^{t}\left[  \sigma_{H}(k_{n}\left(  s\right)
,Y^{n}(k_{n}\left(  s\right)  ))-\sigma_{H}(s,Y^{\infty}(s))\right]
\mathrm{d}B_{s}^{\infty}\right\vert .
\]
Using the estimate (\ref{GHfh}), we have that
\begin{align*}
&  \left\Vert \int_{0}^{t}\left[  \sigma_{H}(s,Y^{n}(s))-\sigma_{H}%
(s,Y^{\infty}(s))\right]  \mathrm{d}B_{s}^{\infty}\right\Vert _{\alpha}\leq\\
&  \leq C\Lambda_{\alpha}(B^{\infty})\\
&  \times\int_{0}^{t}\left(  \left(  t-s\right)  ^{-2\alpha}+s^{-\alpha
}\right)  \left[  \left(  1+\Delta Y^{n}(s)+\Delta Y^{\infty}(s)\right)
\left\Vert Y^{n}(s)-Y^{\infty}(s)\right\Vert _{\alpha}\right]  \mathrm{d}s\\
&  \leq C\Lambda_{\alpha}(B^{\infty})\left\Vert Y^{n}-Y^{\infty}\right\Vert
_{\eta}\left(  1+\left\Vert Y^{n}\right\Vert _{\eta}^{\delta}+\left\Vert
Y^{\infty}\right\Vert _{\eta}^{\delta}\right)  \rightarrow0,
\end{align*}
as $n$ tends to infinity,$P$-a.s. Using the estimate (\ref{Gfa1}) and the
assumptions in (H$\sigma_{H}$), we have \
\begin{align*}
&  \left\Vert \int_{0}^{\cdot}\left(  \sigma_{H}(k_{n}\left(  s\right)
,Y^{n}(k_{n}\left(  s\right)  ))-\sigma_{H}(s,Y^{n}(s)\right)  )\mathrm{d}%
B_{s}^{\infty}\right\Vert _{\infty}\\
&  \leq C\Lambda_{\alpha}(B^{\infty})\left\Vert \sigma_{H}(k_{n}\left(
s\right)  ,Y^{n}(k_{n}\left(  s\right)  ))-\sigma_{H}(s,Y^{n}(s))\right\Vert
_{\alpha,1}\\
&  \leq C\Lambda_{\alpha}(B^{\infty})\left\Vert \sigma_{H}(k_{n}\left(
\cdot\right)  ,Y^{n}(k_{n}\left(  \cdot\right)  ))-\sigma_{H}(\cdot
,Y^{n}(\cdot))\right\Vert _{\infty}^{\varepsilon}\left(  I_{1}+I_{2}%
+I_{3}\right)  ,
\end{align*}
where $\varepsilon>0$ is a small positive number that depends on $\alpha$ and
$\beta$ and we have
\[
I_{1}=\left\Vert \sigma_{H}(k_{n}\left(  \cdot\right)  ,Y^{n}(k_{n}\left(
\cdot\right)  ))-\sigma_{H}(\cdot,Y^{n}(\cdot))\right\Vert _{\infty
}^{1-\varepsilon},
\]%
\begin{align*}
I_{2}  &  =\int_{0}^{T}\int_{0}^{s}\frac{\left\vert \sigma_{H}(s,Y^{n}\left(
s\right)  )-\sigma_{H}(r,Y^{n}(r))\right\vert ^{1-\varepsilon}}{\left(
s-r\right)  ^{\alpha+1}}drds\\
&  \leq C_{1}+C_{2}\int_{0}^{T}\int_{0}^{s}\frac{\left\vert Y^{n}\left(
s\right)  -Y^{n}(r)\right\vert ^{1-\varepsilon}}{\left(  s-r\right)
^{\alpha+1}}drds\\
&  \leq C_{1}+C_{2}\left\Vert Y^{n}\right\Vert _{\eta}^{\left(  1-\varepsilon
\right)  \eta}\leq C,
\end{align*}
where we have used the H\"{o}lder continuity in time in assumption
(H$\sigma_{H}$), and
\begin{align*}
I_{3}  &  =\int_{0}^{T}\int_{0}^{s}\frac{\left\vert \sigma_{H}(k_{n}\left(
s\right)  ,Y^{n}\left(  k_{n}\left(  s\right)  \right)  )-\sigma_{H}%
(k_{n}\left(  r\right)  ,Y^{n}(k_{n}\left(  r\right)  ))\right\vert
^{1-\varepsilon}}{\left(  s-r\right)  ^{\alpha+1}}drds\\
&  \leq C\int_{0}^{T}\int_{0}^{s}\frac{\left\vert k_{n}\left(  s\right)
-k_{n}\left(  r\right)  \right\vert ^{\left(  1-\varepsilon\right)  \beta
}+\left\vert k_{n}\left(  s\right)  -k_{n}\left(  r\right)  \right\vert
^{\left(  1-\varepsilon\right)  \eta}}{\left(  s-r\right)  ^{\alpha+1}}drds
\end{align*}
We can compute this last integral, using the partition on the interval and
decomposing the integrals in finite sums. This integral is uniformly bounded
in $n$. Therefore
\[
\left\Vert \int_{0}^{\cdot}\left(  \sigma_{H}(k_{n}\left(  s\right)
,Y^{n}(k_{n}\left(  s\right)  ))-\sigma_{H}(s,Y^{n}(s)\right)  )\mathrm{d}%
B_{s}^{\infty}\right\Vert _{\infty}\rightarrow0,
\]
as $n$ tends to infinity, $P$-a.s. In order to show the convergence of the
term $A_{1}$, we use again the estimate (\ref{Gfa1}) and Lemmas 7.4 and 7.5 in
\cite{Nualart}. We obtain that
\begin{align*}
A_{1}  &  \leq\left\Vert \sigma_{H}(k_{n}\left(  s\right)  ,Y^{n}(k_{n}\left(
s\right)  ))\right\Vert _{\alpha,1}\Lambda_{\alpha}\left(  B_{s}^{n}%
-B_{s}^{\infty}\right) \\
&  \leq C\left\Vert B^{n}-B^{\infty}\right\Vert _{\infty}^{\varepsilon},
\end{align*}
where $\varepsilon>0$ is a small positive constant which depends on $\alpha$.
Therefore, $A_{1}$ converges to zero as $n$ tends to infinity, $P$-a.s.

\textit{Step 3.- }Recall from step 1 that $(Y^{n},B^{n},W^{n})$ and
$(X^{n},B^{H},W)$ have the same laws. Moreover, $W^{n}$ is a standard Brownian
motion in the appropriate filtration and $B^{n}$ is a fractional Brownian
motion. Therefore, our processes satisfy the stochastic differential
equations
\begin{align*}
Y_{t}^{n}  &  =Y_{0}^{n}+\int_{0}^{t}b(k_{n}(s),Y^{n}(k_{n}(s)))\mathrm{d}s\\
&  +\int_{0}^{t}\sigma_{W}(k_{n}(s),Y^{n}(k_{n}(s)))\mathrm{d}W_{s}^{n}%
+\int_{0}^{t}\sigma_{H}(k_{n}(s),Y^{n}(k_{n}(s)))\mathrm{d}B_{s}^{n}%
\end{align*}
almost surely. So, by step 2, when $n$ tends to infinity, we obtain
\[
Y_{t}^{\infty}=Y_{0}^{\infty}+\int_{0}^{t}b(s,Y_{s}^{\infty})\mathrm{d}%
s+\int_{0}^{t}\sigma_{W}(s,Y_{s}^{\infty}))\mathrm{d}W_{s}^{\infty}+\int
_{0}^{t}\sigma_{H}(s,Y_{s}^{\infty})\mathrm{d}B_{s}^{\infty}.
\]
Therefore, $Y^{\infty}$ satisfies (\ref{sde1}) with \ the driving noises
$W^{\infty}$ and $B^{\infty}$.

The sample paths of \ $Y^{\infty\text{ }}$belong to $C_{0}^{\eta}\ \subset
W_{0}^{\alpha,\infty}$ almost surely, and furthermore, by (\ref{z6}), we have
that
\[
\int_{0}^{T}\mathbb{E}^{W}\left[  \left\Vert Y_{s}^{\infty}\right\Vert
_{\alpha}^{2}\right]  \mathrm{d}s<\infty\text{.}%
\]
Therefore, by Definition \ref{Def: Weak Solution}, $\left(  Y^{\infty
},W^{\infty},B^{\infty}\right)  $ is a weak solution of (\ref{sde1}).
\end{proof}

\medskip We can now proceed with the proof of Theorem \ref{main}.

\medskip

\begin{proof}
[Proof of Theorem \ref{main}]The uniqueness is a consequence of the general
pathwise uniqueness proved in Theorem \ref{Uniqueness} For the existence of a
strong solution we can make use of the classical result by Yamada and Watanabe
\cite{Yamada}, which asserts that pathwise uniqueness and existence of weak
solutions imply the existence of a strong solution. The main difference with
the classical proof is that here we have two random sources independent of the
Wiener process $W$, the initial condition $X_{0}$ and the fractional Brownian
motion $B^{H}$. It suffices to replace $\mathbb{R}^{d}$ by the product space
\ $\mathbb{R}^{d}$ $\times C\left(  [0,T]\right)  ^{m}$, endowed with the
product measure $\mu\times\upsilon$, where $\mu$ is the law of $X_{0}$ and
$\upsilon$ is the law of $B^{H}$ on the space of continuous functions.
\end{proof}

\end{document}